\documentclass[12pt,reqno]{amsart}
\usepackage{mathptmx}     
\usepackage{xcolor}
\usepackage{amssymb}
\usepackage{amscd}
\usepackage{amsmath}
\usepackage{amsthm}
\usepackage{amstext}
\usepackage{amsrefs}
\usepackage[a4paper,tmargin=3cm,bmargin=3cm,lmargin=2cm,rmargin=2cm]{geometry}
\usepackage{hyperref}
\usepackage[colorinlistoftodos]{todonotes}
\usepackage{multicol}

\usepackage{amsrefs}
\usepackage{enumitem}
\usepackage{scrtime}

\newcommand{\R}{\mathbb{R}}
\newcommand{\N}{\mathbb{N}}
\newcommand{\Z}{\mathbb{Z}}

\newcommand{\dd}{\mathrm{d}}
\newcommand{\ddd}{\,\text{\rm{\mbox{\dj}}}}
\newcommand{\bmo}{\mathrm{bmo}}
\newcommand{\BMO}{\mathrm{BMO}}
\newcommand{\SW}{\mathcal{S}}
\newcommand{\m}{\mathtt{m}}

\newcommand{\norm}[1]{\left\Vert#1\right\Vert}
\newcommand{\brkt}[1]{\left(#1\right)}
\newcommand{\abs}[1]{\left|#1\right|}

\newcommand{\esc}[1]{\langle{#1}\rangle}

\newcommand{\Rw}{\mathrm{w}}

    \newtheorem{thm}{Theorem}[section]
    \newtheorem{cor}[thm]{Corollary}
    \newtheorem{lem}[thm]{Lemma}
    \newtheorem{prop}[thm]{Proposition}
    \newtheorem{defn}[thm]{Definition}
    \newtheorem{rem}[thm]{Remark}
    \newtheorem{example}[thm]{Example}

    \numberwithin{equation}{section}

\begin{document}

\title[]{Bilinear pseudodifferential operators with symbol in $BS_{1,1}^m$ on Triebel-Lizorkin spaces with critical Sobolev index}
\subjclass[2000]{}
\author[S. Arias]{Sergi Arias}
\address{Department of Mathematics, Stockholm University, SE-106 91 Stockholm, Sweden}
\email{arias@math.su.se}
\author[S.~Rodr\'iguez-L\'opez]{Salvador Rodr\'iguez-L\'opez}
\address{Department of Mathematics, Stockholm University, SE-106 91 Stockholm, Sweden}
\email{s.rodriguez-lopez@math.su.se}

\thanks{The authors were partially supported by the Spanish Government grant PID2020-
113048GB-I00, funded by MCIN/AEI/10.13039/501100011033.
}


\maketitle

\begin{abstract}
In this paper we obtain new estimates for bilinear pseudodifferential operators with symbol in the class $BS_{1,1}^m$, when both arguments belong to Triebel-Lizorkin spaces of the type $F_{p,q}^{n/p}(\R^n)$. The inequalities are obtained as a consequence of a refinement of the classical Sobolev embedding $F^{n/p}_{p,q}(\R^n)\hookrightarrow\bmo(\R^n)$, where we replace $\bmo(\R^n)$ by an appropriate subspace which contains $L^\infty(\R^n)$. As an application, we study the product of functions on  $F_{p,q}^{n/p}(\R^n)$ when $1<p<\infty$, where those spaces fail to be multiplicative algebras.
\end{abstract}

\section{Introduction}
The classical Sobolev embedding theorem asserts that if $s>n/2$ and $f$ belongs to the inhomogeneous Sobolev space $L^2_s(\R^n)$, that is, if $\esc{D}^s f\in L^2(\R^n)$, then $f$ is a continuous function vanishing at infinity which satisfies \begin{equation}\label{eq:Sobolev_embedding}
\norm{f}_\infty\lesssim  \norm{\esc{D}^s f}_{L^2(\R^n)}.
\end{equation}
The index $n/2$ is called the critical Sobolev exponent, as it is known that such  inequality fails for $s=n/2$, which motivates the terminology.

For the critical exponent, using some classical embeddings between Triebel-Lizorkin and Besov spaces (see \cite{Runst_Sickel}*{Theorem 2.2.2 and Remark 2.2.3/3}), it is possible to replace $L^\infty(\R^n)$ by the space of functions with local bounded mean oscillation, denoted by  $\bmo(\R^n)$. More precisely, it holds that
\begin{equation}\label{emb_bmo_intro}
    \norm{f}_{\bmo(\R^n)} \lesssim \norm{\esc{D}^{n/2} f}_{L^2(\R^n)},
\end{equation}
which is equivalent to having the embedding $L^2_{n/2}(\R^n)\hookrightarrow\bmo(\R^n)$.

More generally, in the scale of Triebel-Lizorkin spaces $F_{p,q}^s(\R^n)$, which recovers the cases above for $p=q=2$, one has the embedding 
\begin{equation}\label{emb_tri_bmo}
    F^{n/p}_{p,q}(\R^n)\hookrightarrow \bmo(\R^n)
\end{equation}
provided $0<q\leq \infty$ and $0<p<\infty$, while
\begin{equation}\label{emb_tri_Linf}
    F^{s}_{p,q}(\R^n)\hookrightarrow L^\infty(\R^n)
\end{equation}
provided $0<q\leq \infty$, and either $s>n/p$ if $1<p<\infty$ or $s=n/p$ if $0<p\leq 1$ (see \cite{Runst_Sickel}*{Theorem 2.2.4/1}).

Under the condition in \eqref{emb_tri_Linf}, $F_{p,q}^s(\R^n)$ is a multiplicative algebra. Indeed, classical bilinear estimates on paraproducts (see \cite{Runst_Sickel}*{Theorem 4.6.4/2}) yield that for 
\[
    s>n\brkt{\frac{1}{\min(1,p)}-1}
\]
one has
\begin{equation}\label{prod_inter}
    \norm{fg}_{F_{p,q}^{s}(\R^n)}\lesssim\norm{f}_{F_{p,q}^{s}(\R^n)}\norm{g}_{L^\infty(\R^n)}+\norm{f}_{L^\infty(\R^n)}\norm{g}_{F_{p,q}^{s}(\R^n)},
\end{equation}
which jointly with the embedding in \eqref{emb_tri_Linf} gives the result.

Following our previous investigations in \cite{Paper2}, the space $L^\infty(\R^n)$ can be replaced in \eqref{prod_inter} by the larger space $X_w(\R^n)$, which is a suitable subspace of $\bmo(\R^n)$ associated to an admissible weight (see Definition \ref{X_w}). More specifically, we obtain a boundedness property for bilinear pseudodifferential operators with symbol $\sigma$ in the class $BS_{1,1}^m$, $m\in\R$ (see Section \ref{cons_and_appl}). Namely, we show in Proposition \ref{BS_diag_for_Xw} that
\begin{equation}\label{estimate_paper2}
    \norm{T_\sigma(f,g)}_{F_{p,q}^{s,1/w}(\R^n)}\lesssim\norm{f}_{F_{p,q}^{s+m}(\R^n)}\norm{g}_{X_w(\R^n)}+\norm{f}_{X_w(\R^n)}\norm{g}_{F_{p,q}^{s+m}(\R^n)}
\end{equation}
holds for $0<p,q\leq\infty$ and \begin{equation}\label{tau_p_q}
    s>\tau_{p,q}:=n\brkt{\frac{1}{\min(1,p,q)}-1}.
\end{equation}

Here $T_\sigma$ is the associated pseudodifferential bilinear operator (see \eqref{eq:bilineal_psi} below) and $F_{p,q}^{s,1/w}(\R^n)$ denotes a Triebel-Lizorkin space of generalised smoothness (see Definition \ref{GSTL} below) associated to an admissible weight $w$. Bilinear pseudodifferential operators with symbols in the bilinear H\"ormander classes have been widely studied by several authors (see for instance \cites{ben-nah-tor,ben-tor,Grafakos-Oh,Gra-Mal_Nai,koe-tom,Naibo,Naibo_notices,Naibo-Thomson} and references therein).

We shall point out that, when $\sigma\equiv 1$ and $w\equiv 1$, then $T_\sigma$ becomes the product of two functions, the space $X_w(\R^n)$ coincides with $L^\infty(\R^n)$, and  $F_{p,q}^{s,1/w}(\R^n)$ coincides with the classical Triebel-Lizorkin space $F_{p,q}^s(\R^n)$ (see Proposition \ref{xw_prop} below). Therefore, the estimate in \eqref{estimate_paper2} recovers the classical inequality in \eqref{prod_inter} for $q\geq \max(p,1)$.

For the particular choice $w(t)=(1+\log_+ 1/t)$, the space $X_w(\R^n)$ becomes $\bmo(\R^n)$  (see Proposition \ref{xw_prop}). Hence, with that choice of $w$ and letting $\sigma\equiv 1$, one can apply \eqref{estimate_paper2} together with the Sobolev embedding in  \eqref{emb_tri_bmo} to obtain a logarithmically subcritical  estimate for the product of functions in $F^{n/p}_{p,q}(\R^n)$, given by
\begin{align*}
    \norm{fg}_{F_{p,q}^{n/p,1/(1+\log_+ 1/t)}(\R^n)}\lesssim\norm{f}_{F_{p,q}^{n/p}(\R^n)}\norm{g}_{F_{p,q}^{n/p}(\R^n)}. 
\end{align*}

As it turns out, one can improve this last estimate by finding a smaller target space. More precisely, a suitable refined version of the Sobolev embedding in \eqref{emb_tri_bmo}, involving the spaces $X_w(\R^n)$ and choosing an appropriate  admissible weight of logarithmic type, will provide us the desired improvement. In that way, we show in Theorem \ref{main_thm} that 
\begin{equation}\label{main_emb_intro}
        F^{n/p}_{p,q}(\R^n)\hookrightarrow X_w(\R^n)
\end{equation}
holds with a particular logarithmic weight $w$ depending on the parameter $p$.

The refined embedding in \eqref{main_emb_intro}, in combination with the estimate in \eqref{estimate_paper2}, allows to produce new boundedness properties for bilinear pseudodifferential operators with symbol in the class $BS_{1,1}^m$. More specifically, we obtain in Theorem \ref{cor_BS} an estimate for $T_\sigma(f,g)$ when both functions belong to the intersection of two Triebel-Lizorkin spaces, one of them with critical Sobolev index.

In that direction, for the case  $\sigma\equiv 1$, we obtain new estimates on the product of functions in $F_{p,q}^{n/p}(\R^n)$ when $1<p\leq\infty$. We shall point out that those spaces are known to be a multiplication algebra if, and only if, $0<p\leq 1$ (see, for instance, \cite{Runst_Sickel}*{Theorem 4.6.4/1}). Given two functions in $F_{p,q}^{n/p}(\R^n)$, we show in Corollary \ref{Prod_Trie} that the product of their norms is larger than the norm of the product in the space $F_{p,q}^{n/p,1/w}(\R^n)$, with $w$ lying in the scale of weights of the form $w(t)=(1+\log_+1/t)^{\alpha}$, with $\alpha=1/r'$, where $r=\max\{ 1,p\}$ and $r'$ denotes its conjugate H\"older exponent. In addition, we also show in Proposition \ref{prop:sharpness} that, in the generality the estimate is stated, the result is sharp in the scales of Triebel-Lizorkin spaces of generalised smoothness with weight of the form $w(t)=(1+\log_+1/t)^{\alpha}$, $\alpha\in\R$. More specifically, we show that for $p=2$, the estimate does not hold for $\alpha<1/2$.

Triebel-Lizorkin spaces of generalised smoothness $F_{p,q}^{s,w}(\R^n)$ with $p<\infty$, and those of logarithmic smoothness in particular, have been previously studied by several authors \cites{cae-mou,mou,dominguez2018function}. To the best of our knowledge, the definition for generalized smoothness of the spaces $F_{\infty,q}^{s,w}(\R^n)$ appeared in  \cite{Paper2}. We prove in this article that the definition of these spaces
is independent on the underlying resolution of unity (see Proposition \ref{ind_res_uni}) and we include some other properties that were of need in the present work (see Propositions \ref{Lifting_for_infty} and \ref{lifting_inverse_weight}).

Finally, we state in Corollary \ref{non_linear_pdes} and Corollary \ref{log_sch_op} two applications of our study to some PDE. The first one deals with well-posedness of some families of non-linear equation on $L^2_{n/2}(\R^n)$, related to some classical equations in mathematical physics, such as the Schr\"odinger equation, the biharmonic Schr\"odinger equation, the half-wave equation or the fractional Schr\"odinger equation. In the second corollary we obtain regularity results to PDEs related to the logarithmic Schr\"odinger operator.

The paper is organised as follows. In Section \ref{prel} we give some preliminaries, such as notations, definitions and properties of function spaces that will be used in the paper. In Section \ref{main} we state and prove a refined Sobolev embedding. Section \ref{cons_and_appl} contains the estimates that we obtain for bilinear pseudodifferential operators, as well as some of its consequences.

\section{Preliminaries}\label{prel}
The notation $A\lesssim B$ will be used to indicate the existence of a constant $C>0$ such that $A\leq C B$. Similarly, we will write $A\thickapprox B$ if both $A\lesssim B$ and $B\lesssim A$ hold. We will also use the notation $\langle\xi\rangle=(1+\vert\xi\vert^2)^{1/2}$ for $\xi\in\R^n$.

The space of Schwartz functions will be denoted by $\SW(\R^n)$ and its topological dual, the space of tempered distributions, by $\SW'(\R^n)$. For a function $f\in \SW(\R^n)$ we define its Fourier transform as
\[
    \mathcal{F}[f](\xi)=\widehat{f}(\xi)=\int_{\R^n}f(x)e^{-i x\xi}\dd x
\]
and we will write
\[
    a(tD)f(x)=\int_{\R^n}a(t\xi)\widehat{f}(\xi)e^{i x\xi}\ddd\xi
\]
for appropriate symbols $a$, or simply $a(D)$ when $t=1$. Here $\ddd \xi$ denotes the normalised Lebesgue measure in $\R^n$ given by $\ddd \xi=(2\pi)^{-n}\dd \xi$.

\subsection{Spaces of bounded mean oscillation} 
The space of functions with bounded mean oscillation, $\BMO(\R^n)$, is the set of all those locally integrable functions $f$ defined on $\R^n$ for which
\[
    \norm{f}_{\BMO(\R^n)}:=\sup_Q\frac{1}{\vert Q\vert}\int_Q\vert f(x)-f_Q\vert\dd x<\infty.
\]
The supremum is taken over all cubes in $\R^n$ whose sides are parallel to the axis, while $\vert Q\vert$ denotes the Lebesgue measure of the cube $Q$ and $f_Q$ is the average of $f$ over $Q$, namely $f_Q=\frac{1}{\vert Q\vert}\int_Qf(x)\dd x$.

The local version of $\BMO(\R^n)$ was considered by D. Goldberg in \cite{Goldberg}, and it will be denoted by $\bmo(\R^n)$. It is defined to be the set of all locally integrable functions $f$ on $\R^n$ for which
\begin{equation}\label{bmo_def}
    \norm{f}_{\bmo(\R^n)}:=\sup_{\ell(Q)<1}\frac{1}{\vert Q\vert}\int_Q\vert f(x)-f_Q\vert\dd x + \sup_{\ell(Q)\geq 1}\frac{1}{\vert Q\vert}\int_Q\vert f(x)\vert\dd x<\infty.
\end{equation}
Here $\ell(Q)$ denotes the side length of the cube $Q$. The function space $\bmo(\R^n)$ is continuously embedded in $\BMO(\R^n)$.

\subsection{Admissible weights and related spaces}

Let us start by introducing the function spaces $X_w(\R^n)$, first defined in \cite{Salva_Wolf}*{Definition 4.2}. Those spaces have played an important role in our previous investigations \cites{Paper1,Paper2}
\begin{defn}\label{X_w}
    Let $w:(0,\infty)\rightarrow (0,\infty)$ be a function satisfying the following properties:
        \begin{enumerate}[label=\Roman*)]
            \item \label{I}For every {compact} interval $I\subseteq(0,\infty)$ we have that
                  \[
                    0<\inf_{t\in I}\left(\inf_{s>0}\frac{w(st)}{w(s)}\right)\leq\sup_{t\in I}\left(\sup_{s>0}\frac{w(st)}{w(s)}\right)<\infty;
                  \]
            \item \label{II} There exists $N>0$ such that $\sup_{t>0}w(t)(1+1/t)^{-N}<\infty$;
            \item \label{III} $\inf_{t>0} w(t)>0$.
        \end{enumerate}
    Let $\phi$ be a Schwartz function supported in the ball $\{\abs{\xi}\leq 2\}$ which is identically one on $\{\abs{\xi}\leq 1\}$. Then $X_w(\R^n)$ is defined to be the set of all locally integrable functions $f$ for which
    \[
        \norm{f}_{X_w(\R^n)}:=\norm{f}_{\BMO(\R^n)}+\sup_{t>0}\frac{\norm{\phi(tD)f}_\infty}{w(t)}<\infty.
    \]
\end{defn}

Motivated by the following proposition, we may think about $X_w(\R^n)$ as intermediate spaces lying in between $L^\infty(\R^n)$ and $\bmo(\R^n)$.

\begin{prop}\label{xw_prop}
    \cite{Paper1}*{Proposition 2.6} Let $w$ and $\phi$ be as in Definition \ref{X_w}. 
    \begin{enumerate}[label=\alph*)]
        \item The definition of the space $X_w(\R^n)$ does not depend on the different choices of function $\phi$, in the sense that different choices induce equivalent norms.
        \item The embeddings $L^\infty(\R^n)\subset X_w(\R^n)\subset \bmo(\R^n)$ hold.
        \item If $w\approx 1$, then $X_w(\R^n)=L^\infty(\R^n)$ with equivalent norms.
        \item \label{d} For $w(t)=1+\log_+ 1/t$, we have that $X_w(\R^n)=\bmo(\R^n)$ with equivalent norms.
    \end{enumerate}
\end{prop}

We shall notice as well that, using the properties in \ref{I},  \ref{II} and \ref{III}, it holds that for any $0<c_1\leq c_2$, there exist $0<d_1\leq d_2$ such that
    \begin{equation}\label{comparable_values_w}
         c_1\leq \frac{t}{s}\leq c_2\quad \mbox{implies that}\quad  d_1\leq \frac{w(t)}{w(s)}\leq d_2.
    \end{equation}

In this paper we will use the terminology \textit{admissible weight} for the following type of functions, slightly modifying the original definition of A. Caetano and S. Moura in \cite{cae-mou}*{Definition 2.1}.

\begin{defn}\label{adm_wei}
    Let $w:(0,1]\rightarrow (0,\infty)$ be a monotonic function, and extend it to $w:(0,\infty)\to (0,\infty)$ by defining $w(t)=w(1)$ for all $t\geq 1$. We say that $w$ is an \textit{admissible} weight if there exist $c,d>0$ such that for all $j\in\N$
    \begin{equation*}
        cw(2^{-j})\leq  w(2^{-2j}) \leq d w(2^{-j}).
    \end{equation*}
\end{defn}

\begin{example}\label{Prototype_admissible}
    A kind of functions satisfying the requirements for an admissible weight could be those of the form
    \[
        w(t):=(1+\log_+1/t)^\lambda\brkt{1+\log (1+\log_+1/t)}^\mu,
    \]
    with $\lambda$, $\mu\in \R$ and $\lambda\cdot\mu\geq 0$.
\end{example}

Admissible weights satisify the following property when evaluated on dyadic numbers, which will be useful later.

\begin{prop}\label{comp_weights}
    \cite{mou}*{Proposition 1.4} For any admissible weight $w$ there exist constants $C_1,C_2>0$ and $b\geq 0$ such that
    \[
        C_1(1+j-k)^{-b}\leq\frac{w(2^{-j})}{w(2^{-k})}\leq C_2(1+j-k)^b
    \]
    for any non-negative integers $j\geq k$.
\end{prop}

We shall point out that the class of admissible weights introduced above do not coincide, in general, with those in Definition \ref{X_w}. However, they satisfy similar properties and both classes coincide under some restrictions. 
\begin{lem}
    \cite{Paper1}*{Lemma 2.12} Let $w$ be an admissible weight. Then $w$ satisfies \ref{I} and \ref{II} in Definition \ref{X_w} above. Moreover, condition \ref{III} holds for an admissible weight $w$ if, and only if, $w$ is either non-increasing, or satisfies that for all $t>0$, $w(t)\approx 1$.
\end{lem}

\subsection{Spaces of generalised smoothness.}

Let $\varphi_0$ be a positive and radially decreasing Schwartz function, supported in the ball $\lbrace\abs{\xi}\leq 3/2\rbrace$, which is identically one on $\lbrace\abs{\xi}\leq 1\rbrace$. We define then $\varphi(\xi):=\varphi_0(\xi)-\varphi_0(2\xi)$ and $\varphi_j(\xi):=\varphi(2^{-j}\xi)$ for $\xi\in\R^n$ and all integers $j\geq 1$. We notice that $\varphi_j$ is supported in the annulus $\lbrace 2^{j-1}\leq\abs{\xi}\leq 2^{j+1}\rbrace$ for all $j\geq 1$ and it holds that $\sum_{j=0}^\infty\varphi_j(\xi)=1$ for all $\xi\in\R^n$. Such family $\{\varphi_j\}_{j\geq 0}$ forms a resolution of unity. 

\begin{defn}\label{GSTL}
    Let $s\in\R$, $0<p\leq\infty$ and $0<q\leq\infty$. Let $w$ be an admissible weight and let $\lbrace\varphi_j\rbrace_{j=0}^\infty$ be a resolution of unity as above. 
    \begin{itemize}
        \item We define the Besov space, $B_{p,q}^s(\R^n)$, to be the set of all tempered distributions $f$ for which
        \[
            \norm{f}_{B_{p,q}^s(\R^n)}:=\left(\sum_{j=0}^\infty 2^{jsq}\norm{\varphi_j(D)f}_{L^p(\R^n)}^q\right)^{1/q}<\infty,
        \]
        with the usual modification if $p=\infty$ or $q=\infty$.
        \item \cites{cae-mou,mou} If $0<p<\infty$, we define the Triebel-Lizorkin space of generalised smoothness, $F_{p,q}^{s,w}(\R^n)$, to be the set of all tempered distributions $f$ for which
        \[
            \norm{f}_{F_{p,q}^{s,w}(\R^n)}:=\norm{\left(\sum_{j=0}^\infty 2^{jsq}w(2^{-j})^q\abs{\varphi_j(D)f}^q\right)^{1/q}}_{L^p(\R^n)}<\infty,
        \]
        with the usual modification if $q=\infty$.
        \item \cite{Paper2}*{Definition 2.7} Let $\mathcal{D}$ be the set of all dyadic cubes in $\R^n$ and $0<q<\infty$. We define $F_{\infty,q}^{s,w}(\R^n)$ to be the set of all tempered distributions $f$ for which
        \begin{align*}
            \norm{f}_{F_{\infty,q}^{s,w}(\R^n)}&:=\norm{\varphi_0(D)f}_\infty\\
            &\quad+\sup_{\substack{Q\in\mathcal{D} \\ \ell(Q)\leq 1}}\left(\frac{1}{\abs{Q}}\int_Q\sum_{j=-\log_2 \ell(Q)}^\infty 2^{sj q}w(2^{-j})^q\abs{\varphi_j(D)f(x)}^q \dd x\right)^{1/q}
        \end{align*}
        is finite.
    \end{itemize}
\end{defn}

\begin{rem}
    If we consider the admissible weight given by the constant function $w\equiv 1$ then Definition~\ref{GSTL} reduces to the classical Triebel-Lizorkin spaces $F_{p,q}^s(\R^n)$. In addition, it is also known that the identities $F_{\infty,2}^0(\R^n)=\bmo(\R^n)$ and $F_{p,2}^s(\R^n)=L^p_s(\R^n)$ hold for $1<p<\infty$ and $s\in\R$, in the sense of equivalent norms (see for instance \cite{Trie83}*{Section 2.3.5}).
\end{rem}

{The spaces $F_{\infty,q}^{s,w}(\R^n)$ were previously defined by the authors in \cite{Paper2} and, to the best of our knowledge, they have not appeared in the literature before. For the sake of completeness, as the details were not included there, we provide below a proof of the property that the definition of  these spaces is independent on the underlying resolution of unity.
}

The argument used is similar to the idea applied in the case of the classical Triebel-Lizorkin spaces, making use of an appropriate multiplier theorem, which we state as a lemma, due to B.J. Park.
\begin{lem}\label{lemmapark}
    \cite{Park}*{Lemma E} Let $0<q<\infty$ and $\nu>n/\min\{1,q\}-n/2$. Consider the sequence of functions $\{f_j\}_{j=0}^\infty$ and assume that there is a constant $C>0$ such that the Fourier transform of each $f_j$ is supported in the ball $\{\abs{\xi}\leq C\cdot 2^j\}$. If the sequence $\{m_j\}_{j=0}^\infty$ of multipliers satisfy
    \[
        \sup_{j\geq 0}\norm{m_j}_{L^ 2_\nu(\R^n)}<\infty,
    \]
    then we obtain the estimate
    \begin{align*}
        &\sup_{\substack{Q\in\mathcal{D} \\ \ell(Q)\leq 1}}\left(\frac{1}{\abs{Q}}\int_Q\sum_{j=-\log_2 \ell(Q)}^\infty \abs{m_j(D)f_j(x)}^q \dd x\right)^{1/q}\nonumber\\
        &\quad\lesssim\sup_{j\geq 0}\norm{m_j(2^j\cdot)}_{L^2_\nu(\R^n)}\sup_{\substack{Q\in\mathcal{D} \\ \ell(Q)\leq 1}}\left(\frac{1}{\abs{Q}}\int_Q\sum_{j=-\log_2 \ell(Q)}^\infty \abs{f_j(x)}^q \dd x\right)^{1/q}.
    \end{align*}
\end{lem}

\begin{prop}\label{ind_res_uni}
    Let $s\in\R$, $0<q<\infty$ and let $w$ be an admissible weight. The space $F_{\infty,q}^{s,w}(\R^n)$ is independent of the chosen resolution of unity defining $\norm{\cdot}_{F_{\infty,q}^{s,w}(\R^n)}$, in the sense that different choices of the resolution of unity give equivalent quasi-norms.
\end{prop}
\begin{proof}
    We follow the approach in \cite{Trie83}*{Proposition 2.3.2/1} for the classical Triebel-Lizorkin spaces, using Lemma \ref{lemmapark} in this case as a multiplier theorem.
    
    Set $f\in F_{\infty,q}^{s,w}(\R^n)$ and let $\{\varphi_j\}_{j\geq 0}$ and $\{\psi\}_{j\geq 0}$ be two resolutions of unity as in Definition \ref{GSTL}. If we define $\psi_{-1}\equiv 0$ then we can write
    \begin{equation}\label{varphi_as_psi}
        \varphi_j=\varphi_j\sum_{r=0}^\infty\psi_r=\sum_{r=-1}^1\psi_{j+r}\varphi_j
    \end{equation}
    for any non-negative integer $j$. In particular, we notice that, given $j\geq 0$, the Fourier transform of $\psi_{j+r}(D)f$ is supported in the ball $\{\abs{\xi}\leq 4\cdot 2^j\}$ for any $r\in\{-1,0,1\}$. Hence, given a cube $Q\in\mathcal{D}$ with $\ell(Q)\leq 1$ and a positive integer $\nu>n/\min\{1,q\}-n/2$, we use \eqref{varphi_as_psi} and Lemma \ref{lemmapark} to see that
    \begin{align}\label{equ_ind_1}
        &\left(\frac{1}{\abs{Q}}\int_Q\sum_{j=-\log_2 \ell(Q)}^\infty 2^{sj q}w(2^{-j})^q\abs{\varphi_j(D)f(x)}^q \dd x\right)^{1/q}\nonumber\\
        &\quad\lesssim\sum_{r=-1}^1\left(\frac{1}{\abs{Q}}\int_Q\sum_{j=-\log_2 \ell(Q)}^\infty 2^{sj q}w(2^{-j})^q\abs{\varphi_j(D)[\psi_{j+r}(D)f](x)}^q \dd x\right)^{1/q}\nonumber\\
        &\quad\lesssim\sum_{r=-1}^1\sup_{j\geq 0}\norm{\varphi_j(2^{j+r}\cdot)}_{L^2_\nu(\R^n)}\sup_{\substack{Q\in\mathcal{D} \\ \ell(Q)\leq 1}}\left(\frac{1}{\abs{Q}}\int_Q\sum_{j=-\log_2 \ell(Q)}^\infty 2^{sj q}w(2^{-j})^q\abs{\psi_{j+r}(D)f(x)}^q \dd x\right)^{1/q}.
    \end{align}
    We observe that it is possible to find a constant $C>0$ such that
    \begin{equation}\label{equ_ind_2}
        \sup_{j\geq 0}\norm{\varphi_j(2^{j+r}\cdot)}_{L^2_\nu(\R^n)}\leq C.
    \end{equation}
    Indeed, it holds that
    \[
        \norm{\varphi_j(2^{j+r}\cdot)}_{L^2_\nu(\R^n)}\lesssim\left(\sum_{\abs{\alpha}\leq \nu}\norm{\partial^\alpha\left(\varphi_j(2^{j+r}\cdot)\right)}_{L^2(\R^n)}^2\right)^{1/2}
    \]
    In addition, the right hand side of the last inequality can be seen to be uniformly bounded on $j$ by using the estimate $\abs{\partial^\alpha\varphi_j(x)}\lesssim 2^{-j\abs{\alpha}}$ for any $x\in\R^n$, consequence of the Schwartz condition of $\varphi_j$.
    
    Next, changing variables, we observe that for $r\in\{-1,0,1\}$ it holds
    \begin{align}\label{equ_ind_3}
       &\sup_{\substack{Q\in\mathcal{D} \\ \ell(Q)\leq 1}}\left(\frac{1}{\abs{Q}}\int_Q\sum_{j=-\log_2 \ell(Q)}^\infty 2^{sj q}w(2^{-j})^q\abs{\psi_{j+r}(D)f(x)}^q \dd x\right)^{1/q}\nonumber\\
       &\quad=\sup_{\substack{Q\in\mathcal{D} \\ \ell(Q)\leq 1}}\left(\frac{1}{\abs{Q}}\int_Q\sum_{j=-\log_2 \ell(Q)+r}^\infty 2^{s(j-r) q}w(2^{-j+r})^q\abs{\psi_j(D)f(x)}^q \dd x\right)^{1/q}\nonumber\\ &\quad\lesssim\sup_{\substack{Q\in\mathcal{D} \\ \ell(Q)\leq 1}}\left(\frac{1}{\abs{Q}}\int_Q\sum_{j=-\log_2 \ell(Q)}^\infty 2^{sj q}w(2^{-j})^q\abs{\psi_j(D)f(x)}^q \dd x\right)^{1/q}.
    \end{align}
     The last inequality is immediate for $r=0$. In the case where $r=1$ the last inequality is obtained after applying Proposition \ref{comp_weights} to get the existence of a constant $b\geq 0$ for which the estimate $w(2^{-j+1})\lesssim 2^{-b}w(2^{-j})$ holds, while the case $r=-1$ is obtained similarly but using the estimate $w(2^{-j-1})\lesssim 2^bw(2^{-j})$, which also follows from Proposition \ref{comp_weights}.
    
    Combining the inequalities in \eqref{equ_ind_1}, \eqref{equ_ind_2} and \eqref{equ_ind_3} we obtain that
    \begin{align}\label{equ_ind_fin_1}
        &\sup_{\substack{Q\in\mathcal{D} \\ \ell(Q)\leq 1}}\left(\frac{1}{\abs{Q}}\int_Q\sum_{j=-\log_2 \ell(Q)}^\infty 2^{sj q}w(2^{-j})^q\abs{\varphi_j(D)f(x)}^q \dd x\right)^{1/q}\\
        &\nonumber\quad\lesssim \sup_{\substack{Q\in\mathcal{D} \\ \ell(Q)\leq 1}}\left(\frac{1}{\abs{Q}}\int_Q\sum_{j=-\log_2 \ell(Q)}^\infty 2^{sj q}w(2^{-j})^q\abs{\psi_j(D)f(x)}^q \dd x\right)^{1/q}.
    \end{align}
    Furthermore, using \eqref{varphi_as_psi} for $j=0$, the Minkowski inequality and the Lebesgue Differentiation Theorem yield
    \begin{align}\label{equ_ind_fin_2}
        \norm{\varphi_0(D)f}_\infty&\leq\norm{\varphi_0(D)[\psi_0(D)f]}_\infty+\norm{\varphi_0(D)[\psi_1(D)f]}_\infty\nonumber\\
        &\lesssim\norm{\psi_0(D)f}_\infty+\norm{\psi_1(D)f}_\infty\nonumber\\
        &\lesssim\norm{\psi_0(D)f}_\infty+\sup_{\substack{Q\in\mathcal{D} \\ \ell(Q)\leq 1}}\left(\frac{1}{\abs{Q}}\int_Q\abs{\psi_1(D)f(x)}^q\dd x\right)^{1/q}\nonumber\\
        &\lesssim\norm{\psi_0(D)f}_\infty+\sup_{\substack{Q\in\mathcal{D} \\ \ell(Q)\leq 1}}\left(\frac{1}{\abs{Q}}\int_Q\sum_{j=-\log_2 \ell(Q)}^\infty 2^{sj q}w(2^{-j})^q\abs{\psi_j(D)f(x)}^q \dd x\right)^{1/q}.
    \end{align}
    Joining \eqref{equ_ind_fin_1} and \eqref{equ_ind_fin_2} we see that the norm $\norm{\cdot}_{F_{\infty,q}^{s,w}(\R^n)}$ associated to the family $\{\varphi_j\}_{j=0}^\infty$ is bounded above by a constant times the norm $\norm{\cdot}_{F_{\infty,q}^{s,w}(\R^n)}$ associated to the family $\{\psi_j\}_{j=0}^\infty$. The same argument, interchanging the roles of $\varphi_j$ and $\psi_j$, gives the converse inequality.
\end{proof}

We shall recall some classical embeddings for Besov and Triebel-Lizorkin spaces, which will be used to prove Theorem \ref{main_thm}.
\begin{prop}\label{embed}
    The following embeddings hold:
    \begin{enumerate}[label=\Roman*)]
        \item\label{BesinTri} Let $0<p<\infty$, $0<q\leq\infty$ and $s\in\R$. Then
        \[
            B_{p,\infty}^{s+n/p}(\R^n)\hookrightarrow F_{\infty,q}^s(\R^n).
        \]
         \item\label{TriinBes} Let $0<q,p_1,q_1\leq\infty$, $0<p<p_1$ and $s,s_1\in\R$. Assume in addition that $s-n/p=s_1-n/p_1$. The embedding
        \[
            F^s_{p,q}(\R^n)\hookrightarrow B^{s_1}_{p_1,q_1}(\R^n)
        \]
        holds if, and only if, $p\leq q_1$.
    \end{enumerate}
\end{prop}
\begin{proof}
    The first statement can be found in \cite{Marschall}*{Lemma 16} while the second one is shown in \cite{Sickel_Triebel}*{Theorem 3.2.1}.
\end{proof}

We shall point out that admissible weights might not be regular enough for some purposes, so it is useful to introduce a regularised version of them.
\begin{defn}\label{Regularised_weight}
 \cite{cae-mou} Let $w$ be an admissible weight, and let  $(\varphi_j)_{j\geq 0}$ be as in Definition \ref{GSTL}. We say that the function 
\begin{equation}\label{log_symbol}
    \Rw(\xi)=\sum_{j=0}^\infty w(2^{-j})\varphi_j(\xi),\quad \xi\in\R^n,
\end{equation}
 is the regularisation of $w$ (associated to the resolution of unity $(\varphi_j)_{j\geq 0}$).
\end{defn}

\begin{prop}\label{cae-mou-lifting}
    \cite{cae-mou}*{Lemma 3.1} Let $w$ be an admissible weight. The functions $\Rw$ and $1/\Rw$ are smooth on $\R^n$ and they satisfy the inequalities
        \begin{equation}\label{ecu1sym}
            \abs{(\partial^\alpha\Rw)(\xi)}\lesssim w(1/\langle\xi\rangle)\langle\xi\rangle^{-\abs{\alpha}}
        \end{equation}
        and
        \begin{equation}\label{ecu2sym}
            \abs{\left(\partial^\alpha\left(\frac{1}{\Rw}\right)\right)(\xi)}\lesssim \frac{1}{w(1/\langle\xi\rangle)}\langle\xi\rangle^{-\abs{\alpha}}
        \end{equation}
        for any multi-index $\alpha\in\N^n$ and any $\xi\in \R^n$.
\end{prop}
\begin{rem}\label{reg_equiv_weight}
    Using \eqref{log_symbol} and the fact that $w(1/\abs{\xi})=w(1)$ for all $\abs{\xi}\leq 1$ yields the equivalence
    $\Rw(\xi)\approx w(1/\abs{\xi})\approx w(1/\langle\xi\rangle)$ for any $\xi\neq 0$. This motivates the terminology of regularisation, as $\Rw$ is smooth and also essentially encodes all the pointwise information of $w$.
\end{rem}

Triebel-Lizorkin spaces of generalised smoothness, $F_{p,q}^{s,w}(\R^n)$, satisfy a lifting property involving the underlying weight function. This property was shown by A. Caetano and S. Moura for the case $0<p<\infty$ and we provide the details on the proof for the missing case $p=\infty$.
\begin{prop}\label{Lifting_for_infty}
    Let either $0<p<\infty$ and $0<q\leq \infty$, or $p=\infty$ and $0<q<\infty$ and $s\in\R$. Let $w$ be an admissible weight and denote by $\Rw$ its regularisation given by \eqref{log_symbol}. It holds that
    \[
        \norm{f}_{F_{p,q}^{s,w}(\R^n)}\approx\norm{\Rw(D)f}_{F_{p,q}^s(\R^n)}
    \]
    for all $f\in F_{p,q}^{s,w}(\R^n)$.
\end{prop}
\begin{proof}
    The proof of the case $0<p<\infty$ can be found in \cite{cae-mou}*{Proposition 3.2}. So we reduce ourselves to the case $p=\infty$. 
    
    We shall prove first the estimate
    \begin{align}\label{sup_part_ineq}
        &\sup_{\substack{Q\in\mathcal{D} \\ \ell(Q)\leq 1}}\left(\frac{1}{\abs{Q}}\int_Q\sum_{j=-\log_2 \ell(Q)}^\infty \abs{2^{sj}\varphi_j(D)[\Rw(D)f](x)}^q \dd x\right)^{1/q}\nonumber\\\
        &\quad\lesssim\sup_{\substack{Q\in\mathcal{D}\\ \ell(Q)\leq 1}}\left(\frac{1}{\abs{Q}}\int_Q\sum_{j=-\log_2 \ell(Q)}^\infty 2^{sjq}w(2^{-j})^q\abs{\varphi_j(D)f(x)}^q \dd x\right)^{1/q}.
    \end{align}
    Indeed, let us consider a Schwartz function $M$ supported in the ring $\{1/4\leq\abs{\xi}\leq 4\}$ which is identically one on $\{1/2\leq\abs{\xi}\leq 2\}$. Then we can write
    \begin{equation}\label{varphi_as_m}
        2^{sj}\varphi_j(D)[\Rw(D)f]=m_j(D)[2^{sj}w(2^{-j})\varphi_j(D)f],
    \end{equation}
    where $m_j(\xi):=w(2^{-j})^{-1}\Rw(\xi)M(2^{-j}\xi)$. Hence, given a positive integer $\nu>n/\min\{1,q\}-n/2$,  Lemma \ref{lemmapark} yields
    \begin{align}\label{apply_mult_thm}
        &\sup_{\substack{Q\in\mathcal{D} \\ \ell(Q)\leq 1}}\left(\frac{1}{\abs{Q}}\int_Q\sum_{j=-\log_2 \ell(Q)}^\infty \abs{m_j(D)[2^{sj}w(2^{-j})\varphi_j(D)f](x)}^q \dd x\right)^{1/q}\nonumber\\   &\quad\lesssim\sup_{j\geq 0}\norm{m_j(2^j\cdot)}_{L^2_\nu(\R^n)}\sup_{\substack{Q\in\mathcal{D} \\ \ell(Q)\leq 1}}\left(\frac{1}{\abs{Q}}\int_Q\sum_{j=-\log_2 \ell(Q)}^\infty \abs{2^{sj}w(2^{-j})\varphi_j(D)f(x)}^q \dd x\right)^{1/q}.
    \end{align}
    The Leibniz rule, the property in \eqref{ecu1sym}, the support condition of $M$ and \eqref{comparable_values_w}, yield
    \begin{align*}
        \abs{\partial^\alpha( m_j(2^j\cdot))(\xi)}&\lesssim \frac{2^{j\abs{\alpha}}}{w(2^{-j})}\sum_{\alpha_1+\alpha_2=\alpha}\abs{\partial^{\alpha_1}\Rw(2^j\xi)}2^{-j\abs{\alpha_2}}\abs{\partial^{\alpha_2}M(\xi)}\\
        &\lesssim\frac{2^{j\abs{\alpha}}}{w(2^{-j})}\sum_{\alpha_1+\alpha_2=\alpha}w(2^{-j})2^{-j\abs{\alpha_1}-j\abs{\alpha_2}}\abs{\partial^{\alpha_2}M(\xi)}
    \end{align*}
    for any multi-index $\alpha\in\N^n$. Using the previous estimate, we see that
    \begin{equation}\label{sup_mj_finite}
        \sup_{j\geq 0}\norm{m_j(2^j\cdot)}_{L^2_\nu(\R^n)}\lesssim\left(\sum_{\abs{\alpha}\leq \nu}\norm{\partial^\alpha\left(m_j(2^{j}\cdot)\right)}_{L^2(\R^n)}^2\right)^{1/2}\lesssim 1.
    \end{equation}
    The inequality in \eqref{sup_part_ineq} is obtained then by combining the estimates in \eqref{varphi_as_m}, \eqref{apply_mult_thm} and \eqref{sup_mj_finite}.
    
    In addition we have that
    \begin{equation}\label{Term_j=0}
        \norm{\varphi_0(D)[\Rw(D)f]}_\infty\lesssim\norm{\varphi_0(D)f}_\infty.
    \end{equation}
    Indeed, let us consider a Schwartz function $\chi$ supported in the ball $\{\abs{\xi}\leq 2\}$ which is identically one in the support of $\varphi_0$. Then, using the classical multiplier theorem in \cite{Trie83}*{Theorem 1.5.2} we get that, for a positive integer $\nu>n/2$, it holds that
    \[
        \norm{\varphi_0(D)[\Rw(D)f]}_\infty=\norm{(\Rw\chi)(D)[\varphi_0(D)f]}_\infty\lesssim\norm{\Rw\chi}_{L^2_\nu(\R^n)}\norm{\varphi(D)f}_\infty.
    \]
    Since $\Rw$ is constant in the support of $\chi$, we see that
    \[
        \norm{\Rw\chi}_{L^2_\nu(\R^n)}^2\lesssim\sum_{\abs{\alpha}\leq \nu}\int_{\R^n}\abs{\Rw(x)\partial^\alpha\chi(x)}^2\dd x\lesssim 1.
    \]
    The combination of the last two estimates gives \eqref{Term_j=0}. Therefore, joining \eqref{sup_part_ineq} and \eqref{Term_j=0} yields the inequality $\norm{\Rw(D)f}_{F_{\infty,q}^s(\R^n)}\lesssim\norm{f}_{F_{\infty,q}^{s,w}(\R^n)}$.
    
    To get the converse inequality we use a similar argument. This time we can write
    \begin{equation*}
        2^{sj}w(2^{-j})\varphi_j(D)f(x)=\tilde{m}_j(D)[2^{sj}\varphi_j(D)[\Rw(D)f]](x),
    \end{equation*}
    where $\tilde{m}_j(\xi)=w(2^{-j})M(2^{-j}\xi)\Rw(\xi)^{-1}$, with $M$ as above. Therefore, Lemma \ref{lemmapark} yields
    \begin{align}\label{big_ineq_conv}
        &\sup_{\substack{Q\in\mathcal{D} \\ \ell(Q)\leq 1}}\left(\frac{1}{\abs{Q}}\int_Q\sum_{j=-\log_2 \ell(Q)}^\infty \abs{2^{sj}w(2^{-j})\varphi_j(D)f(x)}^q \dd x\right)^{1/q}\nonumber\\   &\quad\lesssim\sup_{j\geq 0}\norm{\tilde{m}_j(2^j\cdot)}_{L^2_\nu(\R^n)}\sup_{\substack{Q\in\mathcal{D} \\ \ell(Q)\leq 1}}\left(\frac{1}{\abs{Q}}\int_Q\sum_{j=-\log_2 \ell(Q)}^\infty \abs{2^{sj}\varphi_j(D)[\Rw(D)f](x)}^q \dd x\right)^{1/q}
    \end{align}
    for a positive integer $\nu>n/\min\{1,q\}-n/2$.
    
    Similarly as it was done to obtain \eqref{sup_mj_finite}, we can show that
    \begin{equation}\label{sup_mult_finite}
        \sup_{j\geq 0}\norm{\tilde{m}_j(2^j\cdot)}_{L^2_\nu(\R^n)}<\infty,
    \end{equation}
    where this time one shall use \eqref{ecu2sym}.
    
    Furthermore, repeating the argument used to show \eqref{Term_j=0}, we see that
    \begin{equation}\label{Term_j=0bis}
        \norm{\varphi_0(D)f}_\infty\lesssim\norm{\chi/\Rw}_{L^2_\nu(\R^n)}\norm{\varphi_0(D)[\Rw(D)f]}_\infty\lesssim\norm{\varphi_0(D)[\Rw(D)f]}_\infty,
    \end{equation}
    for any positive integer $\nu>n/2$, where this time one should take into account that $\Rw$ is constant and non-zero in the support of $\chi$.
    
    Combining \eqref{big_ineq_conv}, \eqref{sup_mult_finite} and \eqref{Term_j=0bis} yields the desired estimate $\norm{f}_{F_{\infty,q}^{s,w}(\R^n)}\lesssim\norm{\Rw(D)f}_{F_{\infty,q}^s(\R^n)}$.
\end{proof}
Next we would like to show how the lifting property is satisfied for the inverse weight $1/w$, having as a multiplier operator the inverse of the regularised version of $w$, that is, $1/\Rw$. To this end, we will need the following lemma.
\begin{lem}\cite{Trie83}*{Theorem 2.3.7}\label{KN_bounded}
    Let $\sigma\in \mathcal{C^\infty}(\R^n)$ be a function such that for all multi-index $\alpha\in\N^n$ it satisfies that
    \begin{equation}\label{S0}
        \sup_{\xi\in\R^n} \esc{\xi}^{\abs{\alpha}}\abs{\partial^\alpha_\xi\sigma(\xi)}<+\infty.        
    \end{equation}
    Then, for all $s\in\R$, if $0<p<\infty$ and $0<q\leq \infty$, or $p=\infty$ and $1<q\leq \infty$, the operator $\sigma(D):F^s_{p,q}(\R^n)\to F^s_{p,q}(\R^n)$ is bounded.
\end{lem}
\begin{prop}\label{lifting_inverse_weight}
    Let $w$ be an admissible weight and set $\Rw$ for its regularisation given by \eqref{log_symbol}. Let $\lambda\in\R\setminus\{0\}$ and denote by $\mathrm{u}$ the regularisation of the admissible weight $w^\lambda$. Given $s\in\R$, if $0<p<\infty$ and $0<q\leq\infty$, or if $p=\infty$ and $1<q\leq\infty$, it holds that
    \[
        \norm{\Rw^\lambda(D)f}_{F_{p,q}^s(\R^n)}\approx\norm{\mathrm{u}(D)f}_{F_{p,q}^s(\R^n)}\approx\norm{f}_{F_{p,q}^{s,w^\lambda}(\R^n)}.
    \]
\end{prop}
\begin{proof}
    First of all, we shall notice that
    \[
        \abs{\partial^\alpha\Rw^\lambda(\xi)}\lesssim w(1/\langle\xi\rangle)^\lambda\langle\xi\rangle^{-\abs{\alpha}}
    \]
    holds for all multi-index $\alpha\in\N^n$ and all $\xi\in\R^n$.
    
    Using the previous estimate jointly with Leibniz rule and \eqref{ecu2sym} it follows that
    \begin{equation}\label{KN_cond}
        \abs{\partial^\alpha\left(\frac{\Rw^\lambda}{\mathrm{u}}\right)(\xi)}\lesssim\frac{w(1/\langle\xi\rangle)^\lambda}{w(1/\langle\xi\rangle)^\lambda}\langle\xi\rangle^{-\abs{\alpha}}=\langle\xi\rangle^{-\abs{\alpha}}
    \end{equation}
    for any multi-index $\alpha\in\N^n$ and any $\xi\in\R^n$.
    
    Next, \eqref{KN_cond} and Lemma \ref{KN_bounded} yield 
    \[
        \norm{\Rw^\lambda(D)f}_{F_{p,q}^s(\R^n)}=\norm{\frac{\Rw^\lambda}{\mathrm{u}}(D)\left[\mathrm{u}(D)f\right]}_{F_{p,q}^s(\R^n)}\lesssim\norm{\mathrm{u}(D)f}_{F_{p,q}^s(\R^n)}.
    \]
    Similarly, one shows that $\mathrm{u}/\Rw^\lambda$ also satisfies the condition in \eqref{KN_cond} and the estimate
    \[
        \norm{\mathrm{u}(D)f}_{F_{p,q}^s(\R^n)}\lesssim\norm{\Rw^\lambda(D)f}_{F_{p,q}^s(\R^n)}
    \]
    is obtained. The first equivalence in the statement is hence shown.
    
    The second equivalence in the statement follows by using Proposition \ref{Lifting_for_infty}.
\end{proof}

\section{A refined Sobolev embedding}\label{main}

\begin{thm}\label{main_thm}
    Let either $0<p<\infty$ and $0<q\leq\infty$, or $p=\infty$ and $0<q\leq 2$. If we set $r:=\max\{1,p\}$ then the embedding
    \[
        F^{n/p}_{p,q}(\R^n)\hookrightarrow X_w(\R^n)
    \]
    holds with $w(t)=(1+\log_+1/t)^{1/r'}$.
\end{thm}

\begin{proof}
    Let us start with the case $p=\infty$. Under this choice of $p$, we notice that $r=\infty$ and $w(t)=1+\log_+1/t$, so that $X_w(\R^n)=\bmo(\R^n)$ (see Proposition \ref{xw_prop} \ref{d}). The statement becomes then
    \[
        F_{\infty,q}^0(\R^n)\hookrightarrow\bmo(\R^n)=F_{\infty,2}^0(\R^n),
    \]
    that follows from the inequality $\norm{\cdot}_{F_{\infty,2}^0(\R^n)}\leq\norm{\cdot}_{F_{\infty,q}^0(\R^n)}$, which can be deduced from the embedding of the sequential spaces $\ell^q(\N)\hookrightarrow\ell^2(\N)$, with $0<q\leq 2$.
    
    Let us focus on the case $0<p<\infty$. We shall check first that the embedding
    \begin{equation}\label{Emb_inter}
        F_{p,q}^{n/p}(\R^n)\hookrightarrow F_{\infty,2}^0(\R^n)\cap B_{\infty,r}^0(\R^n)
    \end{equation}
    holds. To do so, we can combine some known embeddings between Besov and Triebel-Lizorkin spaces. Indeed, applying Proposition \ref{embed} \ref{TriinBes} with $s_1=0$ and $p_1=\infty$ we obtain that
    \begin{equation}\label{equem1}
        F_{p,q}^{n/p}(\R^n)\hookrightarrow B_{\infty,r}^0(\R^n).
    \end{equation}
   Moreover, applying Proposition \ref{embed} \ref{BesinTri} we get that
    \begin{equation}\label{equem2}
        B_{p_1,\infty}^{n/p_1}(\R^n)\hookrightarrow F_{\infty,2}^0(\R^n)
    \end{equation}
    for any $0<p_1<\infty$, while Proposition \ref{embed} \ref{TriinBes} yields
    \begin{equation}\label{equem3}
        F_{p,q}^{n/p}(R^n)\hookrightarrow B_{p_1,\infty}^{n/p_1}(\R^n)
    \end{equation}
    for any $0<p<p_1$.
    
    Combining the embeddings in \eqref{equem2} and \eqref{equem3} we obtain that $ F_{p,q}^{n/p}(R^n)\hookrightarrow F_{\infty,2}^0(\R^n)$ for any $0<p<\infty$, which jointly with \eqref{equem1} gives \eqref{Emb_inter}.
    
    Next we would like to see that
    \begin{equation}\label{Int_in_Xw}
        \bmo(\R^n)\cap B_{\infty,r}^0(\R^n)\hookrightarrow X_{w}(\R^n)
    \end{equation}
    holds with $w(t)=(1+\log_+ 1/t)^{1/r'}$. To that end, let $\phi$ be as in Definition \ref{X_w} and let $\psi\in\mathcal{S}(\R^n)$, where we define $\psi_0:=\phi$ and $\psi_j(x):=\phi(2^{-j}x)-\phi(2^{-j+1}x)$ for any integer $j\geq 1$. The functions $\psi_j$, for $j\geq 1$, are supported in the rings $\{2^{j-1}\leq\abs{\xi}\leq 2^{j+1}\}$ and the identity
    \begin{equation}\label{PasQ}
        \phi(2^{-j}\xi)=\phi(\xi)+\sum_{\ell=1}^j\psi_\ell(\xi)
    \end{equation}
    is satisfied for all $\xi\in\R^n$. We want to show that
    \begin{equation}\label{equ_discrete}
        \norm{\phi(2^{-j}D)f}_\infty\lesssim(1+\log 2^j)^{1/r'}\norm{f}_{B_{\infty,r}^0(\R^n)}
    \end{equation}
    for all $j\in\N$. Indeed, we use \eqref{PasQ}, the definition of $\norm{\cdot}_{B_{\infty,r}^0(\R^n)}$ and H\"older's inequality to get that
    \begin{align*}
        \norm{\phi(2^{-j}D)f}_\infty&\leq\norm{\phi(D)f}_\infty+\sum_{\ell=1}^j\norm{\psi_\ell(D)f}_\infty\\
        &\leq\norm{f}_{B_{\infty,r}^0(\R^n)}+\left(\sum_{\ell=0}^\infty\norm{\psi_\ell(D)f}_\infty^r\right)^{1/r}\left(\sum_{\ell=1}^j 1\right)^{1/r'}\\
        &\lesssim\left(1+j^{1/r'}\right)\norm{f}_{B_{\infty,r}^0(\R^n)}\approx\left(1+\log 2^j\right)^{1/r'}\norm{f}_{B_{\infty,r}^0(\R^n)}.
    \end{align*}
    Next we want to show that for all $t>0$
    \[
        \norm{\phi(tD)f}_\infty\lesssim(1+\log_+ 1/t)^{1/r'}\norm{f}_{B_{\infty,r}^0(\R^n)}.
    \]
    Let us first take  $0<t<1$. Then, there exists an integer $j\geq 0$ such that $2^{-j-1}\leq t<2^{-j}$. We notice that
    \[
        \phi(tD)f=\phi(tD)[\phi(2^{-j-2}D)f]
    \]
    from where, applying Minkowski's inequality and using \eqref{equ_discrete} we get that
    \begin{align}\label{low_part}
        \begin{split}
        \norm{\phi(tD)f}_\infty&\leq\norm{t^{-n}\widehat{\phi}(t^{-1}\cdot)}_{L^1(\R^n)}\norm{\phi(2^{-j-2}D)f}_\infty\\
        &\lesssim\left(1+\log 2^{j+2}\right)^{1/r'}\norm{f}_{B_{\infty,r}^0(R^n)}\\
        &\approx\left(1+\log 1/t\right)^{1/r'}\norm{f}_{B_{\infty,r}^0(R^n)}.
        \end{split}
    \end{align}
    When $t\geq 1$ we notice that $\phi(tD)f=\phi(tD)[\phi(2^{-1}D)f]$ from where, applying Minkowski's inequality and the definition of $\norm{\cdot}_{B_{\infty,r}^0(\R^n)}$, we get that
    \begin{align*}
        \norm{\phi(tD)f}_\infty&\leq\norm{t^{-n}\widehat{\phi}(t^{-1}\cdot)}_{L^1(\R^n)}\norm{\phi(2^{-1}D)f}_\infty\lesssim\norm{f}_{B_{\infty,r}^0(\R^n)}.
    \end{align*}
    Combining this and \eqref{low_part} we deduce that
    \[
        \sup_{t>0}\frac{\norm{\phi(tD)f}_\infty}{(1+\log_+1/t)^{1/r'}}\lesssim\norm{f}_{B_{\infty,r}^0(\R^n)},
    \]
    from where $\norm{f}_{X_w(\R^n)}\lesssim\norm{f}_{\bmo(\R^n)}+\norm{f}_{B_{\infty,r}^0(\R^n)}$, which shows \eqref{Int_in_Xw}.
    
    The statement of the theorem follows by combining \eqref{Emb_inter} and \eqref{Int_in_Xw}.
   
\end{proof}

\begin{rem}
    We notice that, if $r=1$, the weight $w$ is constant and hence $X_w(\R^n)=L^\infty(\R^n)$ (see Proposition \ref{xw_prop}). Therefore Theorem \ref{main_thm} states that $F^{n/p}_{p,q}(\R^n)\hookrightarrow L^\infty(\R^n)$ for all $0<q\leq\infty$ and $0<p\leq 1$. This recovers partially the stronger result of  W. Sickel and H. Triebel \cite{Sickel_Triebel}*{Theorem 3.3.1}, which states that the embedding $F^{n/p}_{p,q}(\R^n)\hookrightarrow L^\infty(\R^n)$ holds if, and only if, $0<q\leq\infty$ and $0<p\leq 1$.
\end{rem}

\section{Bilinear pseudodifferential operators}\label{cons_and_appl}

Given a measurable function $\sigma$ on $\R^{3n}$ we denote by $T_\sigma$ the associated bilinear pseudodifferential operator given by
\begin{equation}\label{eq:bilineal_psi}
    T_\sigma(f,g)(x)=\iint\sigma(x,\xi,\eta)\widehat{f}(\xi)\widehat{g}(\eta)e^{ix(\xi+\eta)}\ddd \xi\ddd\eta,\quad x\in\R^n,\quad f,g\in\SW(\R^n).
\end{equation}
The function $\sigma$ is referred to as the symbol of the bilinear operator $T_\sigma$.

The following theorem, shown in \cite{Paper2}*{Theorem 3.1}, gives an estimate for bilinear pseudodifferential operators on Triebel-Lizorkin spaces and the spaces $X_w(\R^n)$ for a certain elementary type of symbols.
\begin{thm}\label{thm_paper2}
    Let us consider the bilinear operator $T_\sigma$ with $\sigma$ of the form
    
    \begin{equation}\label{main_sym}
        \sigma(x,\xi,\eta)=\sum_{j=0}^\infty \m_j(x)\psi_j(\xi)\phi_j(\eta),
    \end{equation}
    where $\lbrace \m_j\rbrace_{j=0}^\infty$,$\lbrace \psi_j\rbrace_{j=0}^\infty$ and $\lbrace \phi_j\rbrace_{j=0}^\infty$ are collections of smooth functions in $\R^n$ satisfying that for every $N\in\N$ there exists $C_N>0$ such that
    \begin{equation}\label{equ4}
        \norm{\partial^\alpha\m_j}_\infty\leq C_N 2^{j(m+\abs{\alpha})}
    \end{equation}
    for all $\abs{\alpha}\leq N$ and some $m\in\R$, as well as
    \begin{align}\label{equ11}
        &\operatorname{supp}\psi_0\subseteq\lbrace\abs{\xi}\lesssim 1\rbrace,\quad \operatorname{supp}\psi_j\subseteq\lbrace\abs{\xi}\approx 2^j\rbrace \quad \mathrm{for} \quad j\geq 1,\nonumber\\
        &\norm{\partial^\alpha\psi_j}_\infty\leq C_N 2^{-j\abs{\alpha}} \quad \mathrm{for \ all} \quad \abs{\alpha}\leq N,
    \end{align}
    and
    \begin{align}\label{equ6}
        &\operatorname{supp}\phi_j\subseteq\lbrace\abs{\xi}\lesssim 2^j\rbrace \quad \mathrm{for} \quad j\geq 0,\nonumber\\
        &\norm{\partial^\alpha\phi_j}_\infty\leq C_N 2^{-j\abs{\alpha}} \quad \mathrm{for \ all} \quad \abs{\alpha}\leq N.
    \end{align}
    
Let $0<p\leq\infty$, $0<q\leq\infty$, $m\in\R$ and let $s>\tau_{p,q}$, with $\tau_{p,q}$ as in \eqref{tau_p_q}. Given a non-increasing admissible weight $w$ we can find $C>0$ such that
    \[
        \norm{T_\sigma(f,g)}_{F_{p,q}^{s,1/w}(\R^n)}\leq C\norm{f}_{F_{p,q}^{s+m}(\R^n)}\norm{g}_{X_w(\R^n)}
    \]
    for all $f,g\in\SW(\R^n)$.
\end{thm}

We say that a smooth function $\sigma$ defined on  $\R^{3n}$ belong to the class $BS_{1,1}^m=BS_{1,1}^m(\R^n)$, if it satisfies
\[
    \abs{\partial^\alpha_x\partial^\beta_\xi\partial^\gamma_\eta\sigma(x,\xi,\eta)}\leq C_{\alpha,\beta,\gamma}(1+\abs{\xi}+\abs{\eta})^{m+\abs{\alpha}-\abs{\beta}-\abs{\gamma}}
\]
for all $(x,\xi,\eta)\in\R^{3n}$, all multi-indices $\alpha,\beta,\gamma\in\N^n$ and some $C_{\alpha,\beta,\gamma}>0$. For symbols in the class $BS^m_{1,1}$ and $N\in\N$ we shall use the notation
\begin{align*}
    \norm{\sigma}_{BS^m_{1,1;N}}:=\max_{\abs{\alpha},\abs{\beta},\abs{\gamma}\leq N}\left(\sup_{x,\xi,\eta\in\R^n}(1+\abs{\xi}+\abs{\eta})^{-(m+\abs{\alpha}-\abs{\beta}-\abs{\gamma})}\abs{\partial^\alpha_x\partial^\beta_\xi\partial^\gamma_\eta\sigma(x,\xi,\eta)}\right).
\end{align*}

In the following proposition, we obtain an estimate for bilinear pseudodifferential operators with symbol in the class $BS_{1,1}^m$, when both arguments of the operator belong to the intersection of a Triebel-Lizorkin space and a space of the type $X_w(\R^n)$. The idea of the proof is to decompose the symbol into the sum of two elementary symbols as in Theorem \ref{thm_paper2}, following the argument used in \cite{koe-tom}*{Theorem 1.1} by K. Koezuka and N. Tomita (see also \cite{ben-tor}). We provide the details of the argument to make the paper self-contained.
\begin{prop}\label{BS_diag_for_Xw}
    Let $0<p,q\leq\infty$, $s>\tau_{p,q}$ (see \eqref{tau_p_q}), $m\in\R$ and  $\sigma\in BS_{1,1}^m$. Consider two admissible weights $u$, $v$ with $u\leq v$. Then we can find a constant $C>0$ and a positive integer $N$ such that
    \[
        \norm{T_\sigma(f,g)}_{F_{p,q}^{s,1/v}(\R^n)}\leq C\norm{\sigma}_{BS_{1,1;N}^m}\left(\norm{f}_{F_{p,q}^{s+m}(\R^n)}\norm{g}_{X_u(\R^n)}+\norm{f}_{X_v(\R^n)}\norm{g}_{F_{p,q}^{s+m}(\R^n)}\right)
    \]
    for all $f,g\in\SW(\R^n)$.
\end{prop}

\begin{proof}

    Let us consider a symbol $\sigma$ in $BS_{1,1}^m$ and let $\lbrace \varphi_j\rbrace_{j=0}^\infty$ be as in Definition \ref{GSTL}. We write $\sigma=\sigma^0+\sigma^1$ with
    \[
        \sigma^0(x,\xi,\eta):=\sum_{j=0}^\infty\sigma^0_j(x,\xi,\eta) \quad \mathrm{and} \quad \sigma^1(x,\xi,\eta):=\sum_{k=1}^\infty\sigma^1_k(x,\xi,\eta),
    \]
    where
    \[
        \sigma^0_j(x,\xi,\eta)=\sum_{k=0}^j\sigma(x,\xi,\eta)\varphi_j(\xi)\varphi_k(\eta)=\sigma(x,\xi,\eta)\varphi_j(\xi)\varphi_0(2^{-j}\eta), \quad j\geq 0,
    \]
    and
    \[
        \sigma^1_k(x,\xi,\eta)=\sum_{j=0}^{k-1}\sigma(x,\xi,\eta)\varphi_j(\xi)\varphi_k(\eta)=\sigma(x,\xi,\eta)\varphi_0(2^{-k+1}\xi)\varphi_k(\eta), \quad k\geq 1.
    \]
    
    Next let us take $\chi_0,\chi\in\SW(\R^n)$ with $\chi_0$ supported in the ball $\lbrace\abs{\xi}\leq 3\rbrace$ and $\chi$ supported in the ring $\lbrace 1/3\leq\abs{\xi}\leq 3\rbrace$, with $\chi_0$ and $\chi$ being identically one on $\lbrace\abs{\xi}\leq 2\rbrace$ and $\lbrace 1/2\leq\abs{\xi}\leq 2\rbrace$ respectively. Define $\chi_j(\xi):=\chi(2^{-j}\xi)$ for $j\geq 1$, in such a way that the functions $\chi_j$ are identically one in the support of $\varphi_j$ for all $j\geq 0$.
    
    Fixed an integer $j\geq 0$ and $x\in\R^n$, we consider the Fourier series expansion of the compactly supported function
    \begin{equation}\label{fou_ser}
        (\xi,\eta)\mapsto\sigma(x,2^j\xi,2^j\eta)\chi(\xi)\chi_0(\eta).
    \end{equation}
    We can describe $\sigma^0_j$ by
    \[
        \sigma^0_j(x,\xi,\eta)=\sum_{k,\ell\in\Z^n}c_{j,k,\ell}(x)e^{ik(2^{-j}\xi)}\varphi_j(\xi)e^{i\ell(2^{-j}\eta)}\varphi_0(2^{-j}\eta),
    \]
    where $c_{j,k,\ell}(x)$ are the Fourier coefficients of the function in \eqref{fou_ser}, which are given by
    \[
        c_{j,k,\ell}(x)=\frac{1}{(2\pi)^n}\iint\sigma(x,2^j\xi,2^j\eta)\chi(\xi)\chi_0(\eta)e^{-i(k\xi+\ell\eta)}\dd\xi\dd\eta,
    \]
    with $\chi_0(\xi)$ instead of $\chi(\xi)$ when $j=0$.
    
    In addition, we observe that the Fourier coefficients of the function in \eqref{fou_ser} and their derivatives can be bounded uniformly in $j$, $k$ and $\ell$. More precisely, given positive integers $a,b,N$, an integration by parts argument allows us to find $N'\in\N$ such that
    \begin{equation}\label{eqcoef}
        \sup_{j\geq 0,k,\ell\in\Z^n}2^{-j(m+\abs{\alpha})}(1+\abs{k})^a(1+\abs{\ell})^b\norm{\partial^\alpha c_{j,k,\ell}}_\infty\lesssim\norm{\sigma}_{BS^m_{1,1;N'}}
    \end{equation}
    for all multi-indices $\abs{\alpha}\leq N$.
    
    For any positive integers $a',a'',b',b''$ we can write $\sigma_j^0$ as
    \[
        \sum_{k,\ell\in\Z^n}(1+\abs{k})^{-a'}(1+\abs{\ell})^{-b'}\m_j^{(k,\ell)}(x)\psi_j^{(k)}(\xi)\phi_j^{(\ell)}(\eta)
    \]
    where
    \begin{align*}
        &\m_j^{(k,\ell)}(x)=(1+\abs{k})^{a'+a''}(1+\abs{\ell})^{b'+b''}c_{j,k,\ell}(x),\\
        &\psi_j^{(k)}(\xi)=(1+\abs{k})^{-a''}e^{ik(2^{-j}\xi)}\varphi_j(\xi),\\
        &\phi_j^{(\ell)}(\eta)=(1+\abs{\ell})^{-b''}e^{i\ell(2^{-j}\eta)}\varphi_0(2^{-j}\eta),
    \end{align*}
    from where we rewrite the symbol $\sigma^0$ as
    \[
       \sum_{k,\ell\in\Z^n}(1+\abs{k})^{-a'}(1+\abs{\ell})^{-b'}\left(\sum_{j=0}^\infty \m_j^{(k,\ell)}(x)\psi_j^{(k)}(\xi)\phi_j^{(\ell)}(\eta)\right).
    \]
    By using the estimates obtained on the Fourier coefficients in \eqref{eqcoef} it is possible to find, for a given $a,b,N\in\N$, a positive integer $N'$ such that
    \[
        \abs{\partial^\alpha\m_j^{(k,\ell)}(x)}\lesssim 2^{j(m+\abs{\alpha})}(1+\abs{k})^{a'+a''-a}(1+\abs{\ell})^{b'+b''-b}\norm{\sigma}_{BS^m_{1,1;N'}}
    \]
    for all $x\in\R^n$ and $\abs{\alpha}\leq N$. In addition, we can also find $C_N>0$ such that
    \[
        \abs{\partial^\alpha\psi_j^{(k)}(\xi)}\leq C_N 2^{-j\abs{\alpha}}(1+\abs{k})^{-a''+N}
    \]
    and
    \[
        \abs{\partial^\alpha\phi_j^{(\ell)}(\eta)}\leq C_N 2^{-j\abs{\alpha}}(1+\abs{\ell})^{-b''+N}
    \]
    for all $\xi,\eta\in\R^n$ and $\abs{\alpha}\leq N$.  Hence by choosing $a',a'',b',b''$ large enough we can reduce ourselves to the study of a symbol as in Theorem \ref{thm_paper2}. 
    
    Hence, applying Theorem \ref{thm_paper2} with the choice $w=u$ yields the existence of a positive integer $N'$ such that
    \begin{equation*}
        \norm{T_{\sigma^0}(f,g)}_{F_{p,q}^{s,1/u}(\R^n)}\lesssim \norm{\sigma}_{BS_{1,1;N'}^m}\norm{f}_{F_{p,q}^{s+m}(\R^n)}\norm{g}_{X_u(\R^n)}.
    \end{equation*}
    An analogous argument, interchanging the roles of $\xi$ and $\eta$, can be repeated to reduce the study of the symbol $\sigma^1$ to a symbol as in \eqref{main_sym}. We apply Theorem \ref{thm_paper2} with the choice $w=v$ to get the existence of a positive integer $N''$ such that
    \begin{equation*}
        \norm{T_{\sigma^1}(f,g)}_{F_{p,q}^{s,1/v}(\R^n)}\lesssim\norm{\sigma}_{BS_{1,1;N''}^m}\norm{f}_{X_v(\R^n)} \norm{g}_{F_{p,q}^{s+m}(\R^n)}.
    \end{equation*}
    
    Next, using the norm inequality $\norm{f}_{F_{p,q}^{s,1/v}(\R^n)}\leq\norm{f}_{F_{p,q}^{s,1/u}(\R^n)}$, it follows that
    \begin{align*}
        \norm{T_{\sigma}(f,g)}_{F_{p,q}^{s,1/v}(\R^n)}&\lesssim\norm{T_{\sigma^0}(f,g)}_{F_{p,q}^{s,1/u}(\R^n)}+\norm{T_{\sigma^1}(f,g)}_{F_{p,q}^{s,1/v}(\R^n)}\\
        &\lesssim\norm{\sigma}_{BS_{1,1;N}^m}\left(\norm{f}_{F_{p,q}^{s+m}(\R^n)}\norm{g}_{X_u(\R^n)}+\norm{f}_{X_v(\R^n)} \norm{g}_{F_{p,q}^{s+m}(\R^n)}\right),
    \end{align*}
    with $N=\max\{N',N''\}$.
\end{proof}

One of the main consequences that we obtain from Theorem \ref{main_thm}, combined with the previous proposition, is the following boundedness property for bilinear pseudodifferential operators with symbol in the class $BS_{1,1}^m$, on Triebel-Lizorkin spaces.

\begin{thm}\label{cor_BS}
    Let $0<p,q\leq\infty$, $s>\tau_{p,q}$ (see \eqref{tau_p_q}), $m\in\R$ and  $\sigma\in BS_{1,1}^m$. Set either $0<\tilde{p}_i<\infty$ and $0<\tilde{q}_i\leq\infty$ or $\tilde{p}_i=\infty$ and $0<\tilde{q}_i\leq 2$, with $i\in\{1,2\}$. Consider the function $w(t)=(1+\log_+1/t)^{1/r'}$ with $r=\max\{1,\tilde{p}_1,\tilde{p}_2\}$. There exist a constant $C>0$ and a positive integer $N$ such that
    \[
        \norm{T_\sigma(f,g)}_{F_{p,q}^{s,1/w}(\R^n)}\leq C\norm{\sigma}_{BS_{1,1;N}^m}\left(\norm{f}_{F_{p,q}^{s+m}(\R^n)}\norm{g}_{F_{\tilde{p}_1,\tilde{q}_1}^{n/\tilde{p}_1}(\R^n)}+\norm{f}_{F_{\tilde{p}_2,\tilde{q}_2}^{n/\tilde{p}_2}(\R^n)}\norm{g}_{F_{p,q}^{s+m}(\R^n)}\right)
    \]
    for all $f,g\in\SW(\R^n)$.
\end{thm}
\begin{proof}
    As a consequence of Theorem \ref{main_thm}, we obtain the embeddings
    \begin{equation}\label{embed_of_Tri_Xw}
        F_{\tilde{p}_1,\tilde{q}_1}^{n/\tilde{p}_1}(\R^n)\hookrightarrow X_{w_1}(\R^n)\quad \mathrm{and} \quad F_{\tilde{p}_2,\tilde{q}_2}^{n/\tilde{p}_2}(\R^n)\hookrightarrow X_{w_2}(\R^n),
    \end{equation}
    where
    \[
        w_1(t)=(1+\log_+1/t)^{1/\max\{1,\tilde{p}_1\}'} \quad \mathrm{and} \quad w_2(t)=(1+\log_+1/t)^{1/\max\{1,\tilde{p}_2\}'}.
    \]
    Next, we shall notice that
    \[
        w_j(t)\leq 2(1+\log_+1/t)^{1/\max\{1,\tilde{p}_1,\tilde{p}_2\}'}=w(t),
    \]
    for $j=1,2$, so that
    \begin{equation}\label{embed_of_Xw}
        X_{w_1}(\R^n)\hookrightarrow X_w(\R^n) \quad \mathrm{and} \quad X_{w_2}(\R^n)\hookrightarrow X_w(\R^n).
    \end{equation}
    Hence the statement follows by aplying Proposition \ref{BS_diag_for_Xw} with $u=v=w$ jointly with the embeddings in \eqref{embed_of_Xw} and \eqref{embed_of_Tri_Xw}.
\end{proof}

Taking the symbol $\sigma$ to be identically one in the previous theorem, we obtain the following logarithmically subcritical estimate for the product of functions in the Triebel-Lizorkin space with Sobolev-critical exponent.

\begin{cor}\label{Prod_Trie}
    Let $0<p<\infty$ and $0<q\leq \infty$ such that
    \[
        {\min(1,p,q)}>\frac{p}{p+1}.
    \]
    If $r=\max\{ 1,p\}$ then the inequality
    \[
        \norm{fg}_{F_{p,q}^{n/p,1/w}(\R^n)}\lesssim\norm{f}_{F_{p,q}^{n/p}(\R^n)}\norm{g}_{F_{p,q}^{n/p}(\R^n)}
    \]
    holds with $w(t)=(1+\log_+1/t)^{1/r'}$.
\end{cor}
\begin{proof}
    Take $s=n/p$, $m=0$, $\sigma\equiv 1$, $p=\tilde{p}_1=\tilde{p}_2$ and $q=\tilde{q}_1=\tilde{q}_2$ in Theorem \ref{cor_BS}.
\end{proof}

\begin{rem} The result above recovers, for $0<p\leq 1$, the fact that  $F_{p,q}^{n/p}(\R^n)$ is a multiplication algebra. Indeed, it is know that $F_{p,q}^{n/p}(\R^n)$ is an algebra, if and only if $0<p\leq 1$ (see, for instance, \cite{Runst_Sickel}*{Theorem 4.6.4/1}). 

For the case $1<p<\infty$, by Proposition \ref{lifting_inverse_weight}, we show the logarithmically-subcritical estimate
\begin{equation}\label{log_mult}
    \norm{\frac{1}{\Rw}(D)(fg)}_{F_{p,q}^{n/p}(\R^n)}\lesssim\norm{f}_{F_{p,q}^{n/p}(\R^n)}\norm{g}_{F_{p,q}^{n/p}(\R^n)},
\end{equation}
where $\Rw$ denotes the regularisation of $w$ given by \eqref{log_symbol}. Notice that, for all $\epsilon>0$, $\esc{D}^{-\epsilon}{\Rw}(D)$ satisfies \eqref{S0}. This yields
\[
        \norm{fg}_{F_{p,q}^{n/p-\epsilon}(\R^n)}\lesssim\norm{f}_{F_{p,q}^{n/p}(\R^n)}\norm{g}_{F_{p,q}^{n/p}(\R^n)},
\]
which recovers the known multiplication result for Triebel-Lizorkin spaces (see e.g.\cite{BaaskeSchmeisser}*{Proposition 2.3}).
\end{rem}

The following proposition, which is an Euclidean version of \cite{El-Fallah}*{Theorem 1.3.1}, states the sharpness of Theorem \ref{cor_BS}, in the sense that, in general, it is not possible to improve the exponent of the logarithmic weight. 
\begin{prop}\label{prop:sharpness}
    Let $\gamma<1/2$ and consider the weight $w(t):=(1+\log_+1/t)^{-\gamma}$. Then, there exists a function $f\in F^{n/2}_{2,2}(\R^n)$ for which $\norm{f^2}_{F^{n/2,w}_{2,2}(\R^n)}=\infty$. 
\end{prop}
\begin{proof}
    Let $\delta>1/2$ and consider the function
    \[
        f_\delta(x):=\int_{\abs{\xi}\geq e}\frac{1}{\abs{\xi}^n\log^\delta\abs{\xi}}e^{ix\xi}\ddd\xi.
    \]
    We shall notice that $f_\delta=\mathcal{F}^{-1}g_\delta$ with
    \[
        g_\delta(\xi)=\frac{\chi_{\{\abs{\xi}>e\}}(\xi)}{\abs{\xi}^n \log^\delta\abs{\xi}}.
    \]
    By using Plancherel's Theorem, we see that
    \[
        \norm{f_\delta}_{L^2_{n/2}(\R^n)}^2=\int_{\R^n}\langle\xi\rangle^n\abs{g_\delta(\xi)}^2\dd\xi\approx\int_{\abs{\xi}>e}\frac{1}{\abs{\xi}^n\log^{2\delta}\abs{\xi}}\dd\xi\approx\int_1^\infty\frac{1}{r^{2\delta}}\dd r,
    \]
    where the last integral is finite since $\delta>1/2$. We deduce than that $f\in F_{2,2}^{n/2}(\R^n)$.

We observe that, for any $\abs{\xi}\geq 2e$, it holds that
    \begin{align}\label{ecu_conv_f_del}
        (\widehat{f_\delta}\ast\widehat{f_\delta})(\xi)&=\int_{\abs{\eta}>e}\frac{\chi_{\{x\in\R^n:\abs{x}>e\}}(\xi-\eta)}{\abs{\eta}^n\log^\delta\abs{\eta}\abs{\xi-\eta}^n\log^\delta\abs{\xi-\eta}}\dd\eta\nonumber\\
        &\geq\int_{e<\abs{\eta}<\abs{\xi}-e}\frac{1}{\abs{\eta}^n\log^\delta\abs{\eta}\abs{\xi-\eta}^n\log^\delta\abs{\xi-\eta}}\dd\eta\nonumber\\
        &\geq\frac{1}{(2\abs{\xi}-e)^n\log^\delta(2\abs{\xi}-e)}\int_{e<\abs{\eta}<\abs{\xi}-e}\frac{1}{\abs{\eta}^n\log^\delta\abs{\eta}}\dd\eta\nonumber\\
        &\approx\frac{1}{(2\abs{\xi}-e)^n\log^\delta(2\abs{\xi}-e)}\frac{\log^{1-\delta}(\abs{\xi}-e)-1}{1-\delta}.
    \end{align}
    Let us denote by $\Rw$ the regularisation of the admissible weight $w$ given by \eqref{log_symbol}. 
    Next, we use Proposition \ref{Lifting_for_infty}. the identification $F_{2,2}^{n/2}(\R^n)=L^2_{n/2}(\R^n)$, Plancherel's Theorem, Remark \ref{reg_equiv_weight} and \eqref{ecu_conv_f_del} to get that
    \begin{align*}
        \norm{f_\delta^2}^2_{F_{2,2}^{n/2,w}(\R^n)}&\approx\norm{\Rw(D)f_\delta^2}^2_{L^2_{n/2}(\R^n)}=\int_\R\frac{\langle\xi\rangle^n}{\Rw(\xi)^2}\abs{(\widehat{f_\delta}\ast\widehat{f_\delta})(\xi)}^2\dd\xi\\
        &\geq\int_{\abs{\xi}\geq 2e}\frac{\langle\xi\rangle^n}{\Rw^\gamma(\xi)^2}\abs{(\widehat{f_\delta}\ast\widehat{f_\delta})(\xi)}^2\dd\xi\\
        &\approx\int_{\abs{\xi}\geq 2e}\frac{\langle\xi\rangle^n}{(1+\log\xi)^{2\gamma}}\abs{(\widehat{f_\delta}\ast\widehat{f_\delta})(\xi)}^2\dd\xi\\
        &\geq\int_{\abs{\xi}\geq 2e}\frac{\langle\xi\rangle^n}{(1+\log\xi)^{2\gamma}}\left(\frac{\log^{1-\delta}(\abs{\xi}-e)-1}{(1-\delta)(2\abs{\xi}-e)^n\log^\delta(2\abs{\xi}-e)}\right)^2\dd\xi\\
        &\approx\int_{\abs{\xi}\geq 2e}\frac{1}{\abs{\xi}^n}\log^{2-4\delta-2\gamma}\abs{\xi}\dd\xi.
    \end{align*}
    Notice that the last integral diverges if, and only if
    \begin{equation}\label{equ_ind_div}
        2-4\delta-2\gamma+1\geq 0\Leftrightarrow\gamma\leq \frac{3}{2}-2\delta.
    \end{equation}
    In particular, if $\gamma< 1/2$ then it is possible to find $\delta>1/2$ so that \eqref{equ_ind_div} is satisfied.
    \end{proof}

\newpage
\subsection{Some consequences and applications}\label{last_section}
\subsubsection{Local Well-Posedness of some PDE's with logaritmically supercritical non-linearities.}  
Consider the following type of non-linear families of equations 
    \[
        i\partial_t u+\brkt{\sqrt{-\Delta}}^{s}u=\mathcal{N}(u,u)\quad s>0.
    \]
where $\mathcal{N}(u,v)=T_m(T_{\sigma}(u,v))$, with $\sigma\in BS^0_{1,1}(\R^ n)$ and $m$ is a linear Fourier multiplier satisfying that
\[
    m(\xi)\lesssim (1+\log_+ \abs{\xi})^{-1/2}, \quad \xi\in\R^n,
\]

For different values of the parameter $s$, the equation above recovers some important equation in mathematical physics. For instance, in the $s=2$, it becomes the well-known Schr\"odinger equation; for $s=4$, the equation becomes the biharmonic Schr\"odinger equation; for $s=1$, the equation is the half-wave equation; and for $0<s<2$ and $s\neq 1$, the equation is the fractional Schr\"odinger equation. 
\begin{cor}\label{non_linear_pdes}
Let $w(t)=(1+\log_+1/t)^{1/2}$ and denote by $\Rw$ its regularisation given by \eqref{log_symbol}. For all $u_0\in L_{n/2}^2(\R^n)$, there exists $T=T(u_0)>0$ and a unique $u\in \mathcal{C}\brkt{[0,T],L_{n/2}^2(\R^n)}$ solving the IVP
\[
    \begin{cases}
        i\partial_t u+\brkt{\sqrt{-\Delta}}^{s}u=\mathcal{N}(u,u)\\ u(t,0)=u_0.
    \end{cases}
\]
\end{cor}
\begin{proof}
    Using the estimate in \eqref{log_mult}, one can show, by a standard argument, that the operator
    \[
        \mathcal{T}_{u_0}(t,u):=e^{it(\sqrt{-\Delta})^s}u_0-i\int_0^t e^{i(t-r)(\sqrt{-\Delta})^s}\left[\mathcal{N}(u,u)\right]\dd r
    \]
    is a contraction on a ball $B$ of radius $R=R(u_0)$ in $\mathcal{C}\brkt{[-T,T],L_{n/2}^2(\R^n)}$, for $T=T(u_0)$ small enough. The statement follows then by applying the Banach fixed point Theorem.
\end{proof}

\subsubsection{Regularity of global solutions to logarithmic equations}

In this section we study the regularity of solutions to  logarithmic equations related to the operator 
    \[
        I+(I-\Delta)^{\mathrm{log}},
    \]
    where $I$ denotes the identity and $(I-\Delta)^{\mathrm{log}}$ stands for the logarithmic Schr\"odinger operator, whose symbol is given by the function $\log(1+\abs{\xi}^2)$. This last operator has been studied by several authors (see \cite{Feulefack} and the references therein).
    
\begin{cor}\label{log_sch_op}
    Let $0<p<\infty$ and $0<q\leq \infty$ such that
    \[
        {\min(1,p,q)}>\frac{p}{p+1}.
    \]

    Given $m\in\R$, take a symbol $\sigma$ in the class $BS_{1,1}^m$. Consider the equation
    \begin{equation}\label{Forca_Barca}
        v(D)u(x)=T_\sigma(f,g)(x),\quad x\in\R^n,
    \end{equation}
    where $v(\xi):=1+\log(1+\abs{\xi}^2)$. Then the function
    \[
        u=\frac{1}{v}(D)\left[T_\sigma(f,g)\right]
    \]
    solves \eqref{Forca_Barca} and satisfies
    \[
        \norm{u}_{F_{p,q}^{n/p,\omega}(\R^n)}\lesssim\norm{\sigma}_{BS_{1,1;N}^m}\norm{f}_{F_{p,q}^{n/p+m}(\R^n)}\norm{g}_{F_{p,q}^{n/p+m}(\R^n)},
    \]
    for some positive integer $N$, where $\omega(t)=(1+\log_+1/t)^{1/p}$.
\end{cor}
\begin{proof}
    Let us denote by $\Rw$ the regularisation in \eqref{log_symbol} of the admissible weight $w(t)=1+\log_+1/t$. Using the fact that
    \[
        \Rw(\xi)\approx w(1/\abs{\xi})\approx v(\xi),
    \]
    one can show that 
    \begin{equation}\label{equiv_weights}
        \abs{\partial^\alpha\left(\frac{\Rw}{v}\right)(\xi)}\lesssim\langle\xi\rangle^{-\abs{\alpha}}
    \end{equation}
    for all multi-index $\alpha\in\N^n$ and all $\xi\in\R^n$. Therefore by using Proposition \ref{lifting_inverse_weight} and Lemma \ref{KN_bounded} we see that
    \begin{align*}
        \norm{u}_{F_{p,q}^{n/p,\omega}(\R^n)}\approx\norm{\frac{\Rw}{v}(D)\left[\frac{1}{\Rw^{1/p'}}(D)T_\sigma(f,g)\right]}_{F_{p,q}^{n/p}(\R^n)}\lesssim\norm{\frac{1}{\Rw^{1/p'}}(D)T_\sigma(f,g)}_{F_{p,q}^{n/p}(\R^n)}.
    \end{align*}
    Then Proposition \ref{lifting_inverse_weight} and Theorem \ref{cor_BS} yield
    \[
        \norm{\frac{1}{\Rw^{1/p'}}(D)T_\sigma(f,g)}_{F_{p,q}^{n/p}(\R^n)}\lesssim\norm{\sigma}_{BS_{1,1;N}^m}\norm{f}_{F_{p,q}^{n/p+m}(\R^n)}\norm{g}_{F_{p,q}^{n/p+m}(\R^n)},
    \]
    from where the result follows.
\end{proof}

\begin{rem}
    We notice that the property in \eqref{equiv_weights} is also satisfied for the reciprocal quotient, that is,
    \[
        \partial^\alpha\left(\frac{v}{\Rw}\right)(\xi)\lesssim\langle\xi\rangle^{-\abs{\alpha}}
    \]
    for any multi-index $\alpha\in\N^n$ and any $\xi\in\R^n$. Hence, in view of Propositions \ref{Lifting_for_infty} and \ref{lifting_inverse_weight}, we have the norm equivalences
    \[
        \norm{f}_{F_{p,q}^{n/p,w}(\R^n)}\approx\norm{v(D)f}_{F_{p,q}^{n/p}(\R^n)}
    \]
    and
    \[
        \norm{f}_{F_{p,q}^{n/p,1/w}(\R^n)}\approx\norm{\frac{1}{v}(D)f}_{F_{p,q}^{n/p}(\R^n)},
    \]
    for the range of indexes as in Corollary \ref{log_sch_op}.
\end{rem}

\end{document}